\documentclass[11pt]{amsart}
\usepackage{amssymb}
\usepackage{amsmath,amssymb,amsfonts,amsthm,
latexsym, amscd, amsfonts, epsfig, epic}
\usepackage{mathrsfs}
\usepackage{color}
\usepackage{eucal}
\usepackage{eufrak}
\usepackage{xspace}
\usepackage{a4}

\theoremstyle{plain}
\newtheorem{theorem}{Theorem}

\newtheorem{lemma}{Lemma}

\theoremstyle{definition}

\theoremstyle{example}

\theoremstyle{remark}

\numberwithin{equation}{section} \numberwithin{lemma}{section}
\numberwithin{figure}{section} \numberwithin{theorem}{section}
\numberwithin{proposition}{section}

\begin{document}
\begin{center}
{\bf\Large The Expected Order of Saturated RNA Secondary Structures}\\
\vspace{15pt} Emma Yu Jin\footnote{The work of this author has been
supported by the Alexander von Humboldt Foundation by a postdoctoral
research fellowship.} and  Markus E. Nebel\footnote{Author to whom
correspondence should be addressed.}
\end{center}

\begin{center}
Department of Computer Science\\
University of Kaiserslautern,\\
67663 Kaiserslautern, Germany\\
Email: \{jin,nebel\}@cs.uni-kl.de
\end{center}

\centerline{\bf Abstract}

{\small \quad\: Over the last 30 years the development of RNA
secondary structure prediction algorithms have been guided and
inspired by corresponding combinatorial studies where the RNA
molecules are modeled as certain kind of planar graphs. The other
way round, new algorithmic ideas gave rise to interesting
combinatorial problems asking for a deeper understanding of the
structures processed. One such example is the notion {\em order} of
a secondary structure as introduced by Waterman (as a parameter on
graphs) in 1978, which reflects a structure's overall complexity:
Regarding so-called {\em hairpin-loops} as the building blocks of a
secondary structure, the order provides information on the
(balanced) nesting-depth of hairpin-loops and thus on the overall
complexity of the structure. In related prediction algorithms, one
first searches for order $1$ structures, increasing the allowed
order step by step and thus
considering an improved structural complexity in every iteration.\\
Subsequently, Zucker {\it et al.} and Clote introduced a more
realistic combinatorial model for RNA secondary structures, the
so-called {\em saturated secondary structures}. Compared to the
traditional model of Waterman, unpaired nucleotides (vertices) which
are in favorable position for a pairing do not exist, i.e.~no base
pair (edge) can be added without violating at least one restriction
for the graphs. That way, one major shortcoming of the traditional
model has been cleared. However, the resulting model gets much more
challenging from a mathematical point of view. As a consequence, so
far only little is
known about the combinatorics of RNA saturated structures.\\
In this paper we show how it is possible to attack saturated
structures and especially how to analyze their order. This is of
special interest since in the past it has been proven to be one of the most
demanding parameters to address (for the traditional model it has
been an open problem for more than 20 years to find asymptotic
results for the number of structures of given order and similar). We
show the expected order of RNA saturated secondary structures of
size $n$ is $\log_4n\left(1+O\left(\frac{\log_2n}{n}\right)\right)$,
if we select the saturated secondary structure uniformly at random.
Furthermore, the order of saturated secondary structures is sharply
concentrated around its mean. As a consequence saturated structures
and structures in the traditional model behave the same with respect
to the expected order. Thus we may conclude that the traditional
model has already drawn the right picture and conclusions inferred
from it with respect to the order (the overall shape) of a structure remain valid
even if enforcing saturation (at least in expectation).}


{\bf Date}: July 2011

\section{Introduction}
The building blocks of RNA are four different nucleotides
$\{a,c,g,u\}=\Sigma$ which are linked to each other in a linear
fashion. Accordingly, the so-called primary structure of RNA,
i.e.~the linear sequence of building blocks, is modeled as string
over $\Sigma$. In addition, non-neighboring nucleotides have a
second means of binding by which certain combinations of nucleotides
($a-u$, $c-g$ and $g-u$) may form pairs, i.e.~stick to each other.
This gives rise to a 3D folding of the molecule which in many cases
determines its biological function. Each such pair reduces the
so-called free energy of the molecule and the conformation of
minimal free energy is adopted in nature. Today, lab techniques to
determine the primary structure of RNA are cheap and efficient while
determining the 3D structure still is a time-consuming and expensive
task. Accordingly one aims for algorithms to predict the structure
from the sequence. However, even if building on rather simple models
for the free energy, its minimization becomes an ${\mathcal
NP}$-complete problem when allowing arbitrary foldings
\cite{Lyngso}. As a consequence, the set of considered structures is
constrained and so-called secondary structures are considered as the
first step towards understanding RNA biological function. There,
only non-crossing pairings of nucleotides are allowed such that --
ignoring types of nucleotides -- the molecule can be represented as
a planar graph \cite{Waterman:78} (see Figure~\ref{F:secandsat}) or
alternatively by strings over $\{ ., (, )\}$ where a $.$ represents
an unpaired nucleotide and a pair of corresponding brackets
represents two paired nucleotides (the left structure of
Figure~\ref{F:secandsat} is in correspondence with
$((..(((......))).))\,$). Even if computing a structure of minimum
free energy (mfe) becomes efficient for secondary structures
(algorithms with cubic time-bounds are well-known), the empiric
thermodynamic data used are incomplete and erroneous such that
suboptimal solutions need to be taken into consideration
\cite{Matt:99}. Computing the suboptimal structures is not
difficult, however, the number of potentially interesting suboptimal
conformation grows exponentially with the length of the nucleotide
sequence. As one possible solution, Zucker and Sankoff suggested to
restrict secondary structure folding to structures whose stacking
regions (runs of consecutive brackets) extend maximally in both
directions. This led to the definition of {\em saturated} structures
for which no base pair can be added without violating the
restrictions for secondary structures, see Figure~\ref{F:secandsat}.
\begin{figure}[ht]
\centerline{%
\epsfig{file=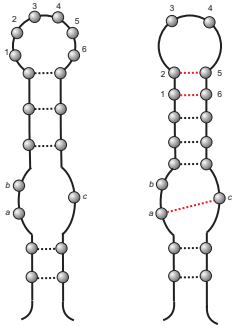,width=0.23\textwidth}\hskip15pt }
\caption{\small Secondary structure (left) and its saturated
counterpart (right) where three additional links have been added
(highlighted in red). The primary structure is given by the chain of
vertices along the solid, pairs of nucleotides are represented by
dotted edges. Note that 3 and 4 cannot be paired since both are
neighbored with respect to primary structure.} \label{F:secandsat}
\end{figure}
Extending the runs of consecutive brackets clears one mayor
shortcoming of the traditional model, i.e.~of secondary structures,
which compared to native molecules tends to have ways too short
stacking regions. Furthermore, in light of the asymptotic number of
saturated structures determined by Clote {\it et al.}
\cite{Clote:09}, the run time of RNA prediction algorithm should be
substantially reduced if the search for suboptimal foldings is
limited to saturated structures only, as observed by
Bompf\"{u}newerer {\it et al.} for so-called canonical structures
\cite{Bomp:08}.

Clote initiated the combinatorial study of saturated structures
\cite{Clote:06} which gets much more challenging than that for
secondary structures from a mathematical point of view. He estimated
the number of saturated structures by applying implicit function
theory to the functional equations of its generating function $S(z)$
\cite{Clote:09}, i.e.,
$$-S(z)^3z^4-S(z)^2z^2(-2+z^2)+S(z)(-1+z^2)+z(1+z)=0,$$
whereas the functional equation for secondary structures is
relatively simple and given by
$$T(z)=z+zT(z)+z^2T(z)+z^2T(z)^2,$$
for $T(z)$ the generating function of secondary structures. Of
course we observe variations of local parameters of the structures
like the length and number of stacking regions or the length and
number of loops (runs of symbols $.$). However, it is not at all
obvious whether saturation has an effect on the overall shape of the
structures. One parameter which allows to measure their overall
shape is the so-called {\em order}, originally introduced by
Waterman in 1978 for algorithmic purposes. A secondary structure $s$
(saturated or not, represented in dot-bracket form) has order $p$ if
we need exactly $p$ iterations of deleting all maximal substrings
$(^k)^k$ within $\phi(s)$ in parallel to find the empty string
$\varepsilon$. Here $\phi$ is the homomorphism implied by
$\phi(()=($, $\phi())=)$ and $\phi(.)=\varepsilon$. Accordingly, the
order provides information on the (balanced) nesting-depth of
so-called hairpin-loops (substring with $\phi$-image $(^n)^n$ which
e.g.~holds for the structures depicted in Figure~\ref{F:secandsat})
and thus on the overall complexity of the structure (it was used by
algorithms to increasingly consider more and more complex foldings
starting with a search space restricted to structures of order $1$).

In this paper we show one way to approach the combinatorics of
saturated structures and especially how to analyze their order. This
-- besides the motivating remarks from above -- is of special
interest since in the past it has been proven to be one of the most
demanding parameters to address (for secondary structures it has
been an open problem for more than 20 years to find asymptotic
results for the number of structures of given order and similar).
For that purpose we discuss the generating function of saturated
structures having order $\ge p$, denoted by $S_p(z)$, from which we
extract the information of the expected order of a saturated
structure of given size. We find that in expectation the order
behaves the same for secondary and saturated structures such that we
may conclude that the traditional model (secondary structures) has
already drawn the right picture and conclusions inferred from it
with respect to the order (the overall shape) of a structure remain
valid even if enforcing saturation (at least in expectation).

The paper is organized as follows. We first present our main
results. Afterwards we describe a streamlined analysis with details
delayed till the last sections (or the appendix).

\section{Main Results}
Let $S(n)$ be the number of saturated RNA secondary structures of
size $n$ and $S_p(n)$ be the number of saturated RNA secondary
structures of size $n$ and having order $\ge p$, then we set $\xi_n$
to be the random variable having probability distribution
$$\mathbb{P}(\xi_n=p)=\frac{S_p(n)-S_{p+1}(n)}{S(n)},$$ namely we
select each saturated structure uniformly at random
among the family of saturated RNA secondary structures of size $n$.
Our main results are summarized as
\begin{theorem}\label{T:exporder}
The expected order of a saturated RNA secondary structure of size $n$ is
$$\mathbb{E}\xi_n=\log_4n\cdot\left(1+O\left(\frac{\log_2n}{n}\right)\right).$$
\end{theorem}
Theorem~\ref{T:exporder} indicates that although the saturation of
secondary structures increases the expected number of paired bases
(and therefore increases the number of hairpin-loops possible) and
scales down the search space, the complexity of the folding
algorithm for saturated structures as given by the order stays
almost the same. We may conclude that the traditional secondary
structure model has already drawn the right picture and conclusions
inferred from it with respect to the order of a structure (its
overall shape) remain valid even if enforcing saturation (at least
in expectation).

Theorem~\ref{T:tail} below proves $\xi_n$ is highly concentrated
around the expected order $\mathbb{E}\xi_n$.
\begin{theorem}\label{T:tail}
Assume we choose $0\le x\le(\frac{1}{2}-\beta)\log_4 n$ for
arbitrary $\beta>0$, then we have
$$\mathbb{P}(|\xi_n-\mathbb{E}\xi_n|\ge x)=O(2^{-x}).$$
\end{theorem}

\section{Road Map of the Proof}
In this section, we shall address the mayor steps and difficulties
of analyzing the expected order of saturated structures by tools
from analytic combinatorics \cite{Drmota:06,Flajoletbook}. We start
by deriving the key recursions for saturated structures of order
$p$.

Let $S(z)$ (resp. $\mathcal{S}$) be the generating function (resp.
the family) of saturated RNA structures and $R(z)$ (resp.
$\mathcal{R}$) be the generating function (resp. the family) of
saturated structures having the first and the last position paired,
i.e., $\mathcal{R}=(\mathcal{S})$ where the parenthesis represents
the paired bases and $R(z)=z^2 S(z)$. Furthermore, let $S_p(z)$ and
$R_p(z)$ represent the corresponding generating function assuming
order $\ge p$, $p\ge 1$. By decomposing the saturated structure into
independent $\mathcal{R}$-type structures, we obtain the functional
equation for $S(z)$
\begin{equation}\label{E:Sat}
S(z)=\sum_{i=0}^{\infty}\left(1+(i+1)(z+z^2)\right)R(z)^{i}-1
=\frac{z^2 S(z)}{1-z^2 S(z)}+\frac{z+z^2}{(1-z^2 S(z))^2}.
\end{equation}
Now, taking the order into account (omitting variable $z$ for the ease of notation),
we find the following recurrences for $S_p$ and $R_{p+1}$, $p\ge 1$,
\begin{eqnarray}
S_p&=&\sum_{i\ge 1}(1+(i+1)(z+z^2))(R^{i}-(R-R_p)^{i})\nonumber\\
\label{E:recSR}
&=&\frac{R_p[1+2z^2+2z-2R-2Rz-2Rz^2+R^2+(1+z+z^2-R)R_p]}{(R-1)^2(R-R_p-1)^2},
\end{eqnarray}
\begin{align}
R_{p+1}&=&R-z^2\left[\sum_{i=1}^{\infty}(1+(i+1)(z+z^2))
\left((R-R_p)^i+(R_p-R_{p+1})i\right.\right.\nonumber\\
&&\left.\left.\times(R-R_{p})^{i-1}\right)+z+z^2\right]\\
\label{E:recR} &=&\frac{(-R-z^2)R_{p}^3+(-3R+3R^2+3Rz^2-z^2)R_{p}^2}
{-R_{p}^3+(3R-3)R_{p}^2+(6R-3-3R^2+z^2)R_{p}+(R-1)P_R},
\end{align}
where $P=R^3+(z^2-2)R^2+(1-z^2)R-z^3-z^4$ and
$P_R=\partial{P}/\partial{R}=3R^2+2(z^2-2)R+(1-z^2)$ and the initial
conditions are $R_1=R$ and $S_0=S$.\\[2mm]
Unlike for secondary structures\footnote{For secondary structures
the expected order has been analyzed by making use of well-known
closed form representations of multivariate generating function for
binary trees having Horton-Strahler number $p$. By the use of
appropriate symbolic substitutions for the different variables the
binary trees with Horton-Strahler number $p$ were expanded into the
secondary structures of order $p$ and a closed form for the
corresponding generating function followed \cite{MN:01}.}, due to
the non-local dependencies imposed for saturation neither the
appropriate symbolic substitution nor the closed form solution of
recurrence (\ref{E:recSR}) could possibly exist, for which we have
to decode the information of expected order from the recurrence
itself other than attempting to solve it. Therefore, the proof for
the expected order of saturated structures consists of locating the
dominant singularities of $S_p(z)$ for $p\ge 0$, verifying the
analytic continuation of $S_p(z)$ for some $\Delta$-domain, which
guarantees the validness of integration along Hankel contour, see
Figure~\ref{F:newhankel}, and finding the singular expansion of
$S_p(z)$ within the intersection of $\Delta$-domain and a small
neighborhood of the dominant singularity. Finally we apply a
transfer theorem on the singular expansions of $S_p(z)$ and $S(z)$
to extract the $n$-th coefficient of $\sum_{p\ge 1}S_p(z)$ and
$S(z)$, and conclude the expected order $\mathbb{E}\xi_n$ via
$$\mathbb{E}\xi_n=\frac{[z^n]\sum_{p\ge 1}S_p(z)}{[z^n]S(z)}.$$
The results on the deviation to the expected order
follows similarly.\\
Before we proceed, we present the Transfer Theorem by Flajolet and
Odlyzko \cite{Flajoletbook}. The central point of this theorem is to
use of Cauchy's formula by integrating along the Hankel contour
depicted in Figure~\ref{F:newhankel}, which is guaranteed by the
analytic continuation within a $\Delta$-domain. We set
$$\Delta_{z_0}(M,\phi)=\{z\,|\,\,|z|<M, z\ne z_0,|\arg(z-z_0)|>\phi\}$$
where $M>z_0$ and $0<\phi<\frac{\pi}{2}$. Let $U_{z_0}(r,\phi)$ be
the intersection of $\Delta_{z_0}(M,\phi)$ and the neighborhood of
$z_0$, i.e.,
$$U_{z_0}(r,\phi)=\{z\,|\,\,0<|z-z_0|<r, |\arg(z-z_0)|>\phi\},$$
then we have:
\begin{theorem}{\bf (Transfer Theorem)\cite{Flajoletbook}}\label{T:transfer}
Assume that $f(z)$ is analytic within $\Delta_{1}(M,\phi)$, and for
$z\in U_{1}(r,\phi)$, $f(z)$ satisfies
$$f(z)=O\left(\sqrt{1-z}\cdot\log_2\left(\frac{1}{1-z}\right)\right).$$
Then we have $[z^n]f(z)=O(n^{-\frac{3}{2}}\cdot\log_2 n)$.
\end{theorem}
Theorem~\ref{T:transfer} assumes the dominant singularity is $z=1$.
However, the case of a dominant singularity at $z=z_0\ne 1$, can
always be boiled down to the case where $z=1$ is the dominant
singularity according to
$$[z^n]f(z)=z_0^n\cdot[z^n]f\left(\frac{z}{z_0}\right).$$
\begin{figure}[ht]
\centerline{%
\epsfig{file=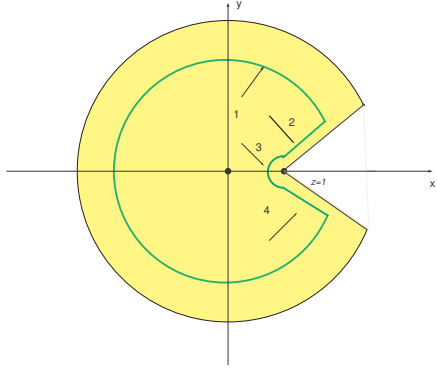,width=0.45\textwidth}\hskip15pt }
\caption{\small $\Delta_1$-domain (yellow) and Hankel contour
(green): Transfer theorem applies Cauchy's formula by integrating
along the Hankel contour, colored in green. The inner incomplete
circle $3$, together with two rectilinear lines $2$ and $4$ mainly
contribute to the integral. Here we assume the dominant singularity
is at $z=1$.} \label{F:newhankel}
\end{figure}
In what follows we detail the steps that are needed for the singularity
analysis of $\sum_{p\ge 1}S_p(z)$.\\[2mm]
{\bf Step $1$: Locate dominant singularities:} We first observe that
the dominant singularity of $S(z)$ is unique since $[z^n]S(z)\ne 0$
holds for arbitrary $n$ and therefore $S(z)$ is aperiodic
\cite{Flajoletbook}. Assume $z_0$ is the unique dominant singularity
of $S(z)$, then $z_0$ is also the unique dominant singularity of
$S_p(z)$ for $p\ge 0$. Indeed, consider the field extension of
the rational function field $Q(z)$ induced by algebraic functions
$S_{p}(z)$, we can inductively prove that $[Q(S_p(z)):Q(z)]=3$ based
on its tower relation
$[Q(S_p(z)):Q(z)]=[Q(S_p(z)):Q(S_{p-1}(z))][Q(S_{p-1}(z)):Q(z)]
=[Q(S_{p-1}(z)):Q(z)]$. In other words, $S_p(z)$ is an algebraic
function of degree $3$ over the field $Q(z)$. Let $S_{\le p}(z)$ be
the generating function of saturated structures having order $\le
p$, similarly we can prove $S_{\le p}(z)$ is rational and in view of
$S_{p}(z)=S(z)-S_{\le p-1}(z)$, we can claim that $S_p(z)$ ($p\ge
0$) have the same unique dominant singularity as $S(z)$. Otherwise,
suppose $z=\gamma<z_0$ is the dominant singularity of $S_p(z)$ and
therefore $S_p(\gamma)<S_p(z_0)<S(z_0)<\infty$, which contradicts to
the fact that $S_p(\gamma)=S(\gamma)-S_{\le p-1}(\gamma)=\infty$
since $S_{\le p-1}(z)$ is a rational function and $z=\gamma$ must be
one of the poles of $S_{\le p-1}(z)$. Furthermore, $z=z_0$ is the
unique dominant singularity of $S_p(z)$ since $[z^n]S_p(z)\ne 0$ and
$S_p(z)$ is aperiodic.
\begin{lemma}\label{L:domiSp}
Let $z_0$ be the unique dominant singularity of $S_p(z)$ $(p\ge 0)$,
then $z_0\approx 0.424687$.
\end{lemma}
We apply the implicit function theorem on eq.~(\ref{E:Sat}) to extract
the unique dominant singularity of $S(z)$, which is also the unique dominant
singularity of $S_p(z)$.\\[2mm]
{\bf Step $2$: Establish the analytic continuation in some
$\Delta_{z_0}$-domain:} Since $S_p(z)$ is an algebraic function of
degree $3$ over the rational function field $Q(z)$, $S_p(z)$ must be D-finite, which
allows for analytic continuation in any $\Delta_{z_0}$-domain
containing zero \cite{Stanley:80}.\\[2mm]
{\bf Step $3$: Singular expansion:} We shall show the singular
expansion of $S_p(z)$ within $U_{z_0}(\epsilon,\phi)$ for
sufficiently small $\epsilon>0$ and $0<\phi<\frac{\pi}{2}$. Our
strategy is to transform the fractional form of the recursion for $R_p(z)$
(eq.~(\ref{E:recR})) into ``linear'' form, based on the
contributions of individual terms to the behavior of $R_{p}(z)$ for
different $p$.\\[2mm]
{\bf Case $1$}: $p\le p_M=\max\left\{p:\left|P_R(R-1)\right|
\le\left|\frac{a_2}{4}\cdot R_{p}\right|\right\}$ for
$a_2=-3R+3R^2+3Rz^2-z^2$.


\begin{lemma}\label{L:expanS}
Assume that $z\in U_{z_0}(\epsilon,\phi)$ and
$a_2=-3R+3R^2+3Rz^2-z^2$, then
\begin{eqnarray*}
S_{p+1}(z)&=&s_p(z_0)\,2^{-p}-\frac{1}{z_0^2}\cdot\frac{P_R(R-1)}{2a_2}\\
&&+O\left(\frac{p}{2^p}\left|\frac{P_R(R-1)}{a_2}\right|\right)+O\left(2^{p}
\left|\frac{P_R(R-1)}{a_2}\right|^2\right).
\end{eqnarray*}
holds for $p\le p_M$ and $s_p(z_0)=s+O(2^{-p})$ where $s>0$ is
constant.
\end{lemma}
\noindent{\bf Case $2$}: $p>p_M$. We continue analyzing the recurrence relations
for $S_{p}(z)$ and $R_p(z)$ for $p>p_M$. Let
$A_p'=-\frac{R+z^2}{a_2}\cdot R_{p}^3$ and
$B_p'=-\frac{1}{a_2}R_{p}^3+\frac{3R-3}{a_2}R_{p}^2
-\frac{3P_RR_{p}}{a_2}$. Note that $A_p'\rightarrow 0$ and
$B_p'\rightarrow 0$ as $p\rightarrow \infty$ and $z\in
U_{z_0}(\epsilon,\phi)$. Then we simply have
\begin{eqnarray}\label{E:simprecR}
R_{p+1}(z)=\frac{R_{p}^2+A_p'}{\frac{P_R(R-1)}{a_2}+2R_{p}+B_p'}.
\end{eqnarray}
We observe that $B_p'$ and $A_p'$ converge to $0$ faster than
$R_{p}$ as $p\rightarrow\infty$, and it only remains to determine
the major contribution between $R_p$ and $\frac{P_R(R-1)}{a_2}$ from
the denominator to the behavior of $R_p$ for different $p$. Here we
all reduce the recursions to the function
$h(x,\mu,\nu)=\frac{x^2+\mu}{1+2x+\nu},$ from which we can prove
$h(x,\mu,\nu)=h(x,0,0)+O(\max\{|\mu|,|\nu|\})$ holds uniformly for
$x\ne \frac{1}{2}$ as $\max\{|\mu|,|\nu|\}\rightarrow 0$. In order
to asymptotically solve eq.~(\ref{E:simprecR}), we need to avoid
$R_p=\frac{1}{2} \frac{P_R(R-1)}{a_2}$, which may occur when $p$ is
sufficiently large. To this aim, we select $\lambda_1>0$ and
$\lambda_2>0$ such that for $p\le p_M+\lambda_2$,
$\left|\frac{P_R(R-1)}{a_2}\right|\le |R_{p}|$ and for $p\ge
p_M-\lambda_1$, $|R_{p}|\ge 8\left|\frac{P_R(R-1)}{a_2}\right|$.
Lemma~\ref{L:phasep} below shows the ``continuity'' of the phase
transition around $p=p_M$.
\begin{lemma}\label{L:phasep}
Assume $z\in U_{z_0}(\epsilon,\phi)$ and $p_0=p_M-\lambda_1$, then
for arbitrary but fixed $\delta\le \lambda_2$, we have uniformly for
$z$ and for $0\le k\le\lambda_1+\delta$ that,
\begin{eqnarray*}
S_{p_0+k}&=&\frac{1}{z_0^2}\frac{\frac{P_R(R-1)}{a_2}}
{\left(\frac{\frac{P_R(R-1)}{a_2}}{R_{p_0}}+1\right)^{2^k}-1}
+O\left(\left|\frac{P_R(R-1)}{a_2}\right|^2\right),
\end{eqnarray*}
where $a_2=-3R+3R^2+3Rz^2-z^2$.
\end{lemma}
\begin{lemma}\label{L:largep}
Assume that $z\in U_{z_0}(\epsilon,\phi)$, there exists
$\kappa_0\ge\lambda_2$ such that for $p>p_M+\kappa_0$,
$$S_{p+1}(z)=O\left(\left|\frac{P_R(R-1)}{a_2}\right|
\exp(-\ln2\cdot 2^p)\right).$$
\end{lemma}
\noindent{\bf Step $5$: Transfer to coefficients:} It only remains to
translate the singular expansion of the function into an asymptotic
estimate of its coefficients.
\begin{theorem}
The expected order of a saturated secondary structures of size $n$ is
$$\mathbb{E}\xi_n=\log_4n\cdot\left(1+O\left(\frac{\log_2n}{n}\right)\right).$$
\end{theorem}
\begin{proof}
We first analyze the expectation function $F(z)=\sum_{p\ge 1}S_p(z)$ for $z\in
U_{z_0}(\epsilon,\phi)$. According to Lemma~\ref{L:expanS},
Lemma~\ref{L:phasep} and Lemma~\ref{L:largep}, we have for $p\ge 1$,
\begin{eqnarray*}
\sum_{p\le p_M}S_{p+1}(z)&=&\sum_{p\le
p_M}s_p(z_0)\,2^{-p}-\frac{p_M}{z_0^2}\cdot\frac{P_R(R-1)}{2a_2}
+O\left(\left|\frac{P_R(R-1)}{a_2}\right|\right)\\
&=&\sum_{p\ge
1}s_p(z_0)\,2^{-p}-\frac{p_M}{z_0^2}\cdot\frac{P_R(R-1)}{2a_2}
+O\left(\left|\frac{P_R(R-1)}{a_2}\right|\right).\\
\sum_{p>p_M}S_{p+1}(z)&=&\sum_{p_M<p\le
p_M+\kappa_0}S_{p+1}(z)+\sum_{p>p_M+\kappa_0}S_{p+1}(z)\\
&=&O\left(\left|\frac{P_R(R-1)}{a_2}\right|\right)
+\sum_{p>p_M+\kappa_0}O\left(\left|\frac{P_R(R-1)}{a_2}\right|\exp(-\ln2\cdot
2^p)\right)\\
&=&O\left(\left|\frac{P_R(R-1)}{a_2}\right|\right).
\end{eqnarray*}
In combination of the cases $p\le p_M$ and $p>p_M$, we obtain
\begin{eqnarray*}
F(z)&=&\sum_{p\le p_M}S_{p+1}(z)+\sum_{p>p_M}S_{p+1}(z)+S_1(z)\\
&=&\sum_{p\ge
0}S_{p+1}(z_0)+\left(S(z)-\frac{1}{1-z}-S_1(z_0)\right)
-\frac{p_M}{z_0^2}\cdot\frac{P_R(R-1)}{2a_2}\\
&&+O\left(\left|\frac{P_R(R-1)}{a_2}\right|\right).
\end{eqnarray*}
Recall that $p_M$ is given by
$$p_M=\max\left\{p:\left|P_R(R-1)\right|
\le\left|\frac{a_2}{4}\cdot R_{p}\right|\right\}$$
and we need to find an appropriate representation for it. For $z\in
U_{z_0}(\epsilon,\phi)$,
$p_M\approx-\log_2\left|\frac{P_R(R-1)}{a_2}\right|$. By setting
$F_0=F(z_0)$ and
$S(z)=S(z_0)-\frac{1}{z_0^2}\frac{P_R(R-1)}{2a_2}+O(z-z_0)$, we
simplify $F(z)$ into
\begin{eqnarray*}
F(z)&=&
F_0-\frac{1}{z_0^2}\frac{P_R(R-1)}{2a_2}-\frac{p_M}{z_0^2}\frac{P_R(R-1)}{2a_2}
+O\left(\sqrt{1-\frac{z}{z_0}}\right)\\
&=&F_0+\frac{\log_2\left|\frac{P_R(R-1)}{a_2}\right|}{z_0^2}\frac{P_R(R-1)}{2a_2}
+O\left(\sqrt{1-\frac{z}{z_0}}\right)\\
&=&F_0+\frac{1}{z_0^2}\frac{P_R(R-1)}{2a_2}\log_2
\left(\frac{P_R(R-1)}{a_2}\right)+O\left(\sqrt{1-\frac{z}{z_0}}\right)\\
&=&F_0-\sqrt{\frac{P_z(z_0)}{2P_{RR}(z_0)z_0^3}}\sqrt{1-\frac{z}{z_0}}
\log_2\left(\frac{1}{1-\frac{z}{z_0}}\right)
+O\left(\sqrt{1-\frac{z}{z_0}}\right).
\end{eqnarray*}
According to Theorem~\ref{T:transfer} (Transfer Theorem), for
$n\ge 1$, the expected order of a saturated secondary structure is
thus given by
\begin{eqnarray*}
\mathbb{E}\xi_n&=&\frac{[z^n]F(z)}{[z^n]S(z)}
=\frac{-n^{-\frac{3}{2}}}{\Gamma(-\frac{1}{2})}\cdot\log_2n\cdot
z_0^{-n}\cdot z_0^2\sqrt{\frac{P_z(z_0)}{2P_{RR}(z_0)z_0^3}}
\sqrt{\frac{2\pi\,P_{RR}(z_0)}{z_0\,P_z(z_0)}}n^{\frac{3}{2}}z_0^n\\
&&\times\left(1+O\left(\frac{\log_2n}{n}\right)\right)\\
&=&\log_4n\cdot\left(1+O\left(\frac{\log_2n}{n}\right)\right),
\end{eqnarray*}
whence the proof is complete.
\end{proof}
Finally we discuss the large deviation of the random variable $\xi_n$.
\begin{theorem}
Assume we choose $0\le x\le(\frac{1}{2}-\beta)\log_4 n$ for
arbitrary $\beta>0$, then we have
$$\mathbb{P}(|\xi_n-\mathbb{E}(\xi_n)|\ge x)=O(2^{-x}).$$
\end{theorem}
\begin{proof}
For $p\le \log_4 n$, Lemma~\ref{L:expanS} in combination with the Transfer
Theorem implies
$$\mathbb{P}(\xi_n\ge p)=1+O\left(\frac{2^p}{\sqrt{n}}\right)
+O\left(\frac{p}{2^p}\right).$$ For $p>\log_4n$,
Lemma~\ref{L:phasep} indicates that
$$\mathbb{P}(\xi_n\ge p)=O\left(\exp\left(-\frac{\beta'2^p}{\sqrt{n}}\right)\right)
\quad\mbox{ for }\beta'>0.$$ Consequently the theorem follows.
\end{proof}


\newpage

\begin{thebibliography}{10}

\bibitem{Bomp:08}
A.F. Bompf\"{u}newerer, R. Backofen, S.H. Bernhart, J. Hertel,
I.L.,Hofacker, P.F. Stadler and S. Will, Variation on RNA folding
and alignment: lessons from Banasque. \textit{J. Math. Biol.} {\bf
56}(1-2) (2008), 129-144.

\bibitem{Cech:86}
T.R. Cech, RNA as an enzyme, \textit{Sci.Am.} {\bf 255}(5) (1986),
64-75.

\bibitem{Clote:06}
P. Clote, Combinatorics of Saturated Secondary Structures of RNA,
\textit{J. Comp. Biol.}, {\bf 13}(9) (2006), 1640-1657.

\bibitem{Clote:09}
P. Clote, E. Kranakis, D. Krizanc and B. Salvy, Asymptotics of
Canonical and Saturated RNA secondary structures, \textit{J.
Bioinformatics and Comp. Biol.}, {\bf 5} (2009), 869-893.

\bibitem{Drmota:06}
M. Drmota and H. Prodinger, The register function for $t$-ary trees,
\textit{ACM Transactions on Algorithms}, {\bf 2} (2006), 318-334.

\bibitem{Flajoletbook}
P. Flajolet and R. Sedgewick, Analytic combinatorics,
ISBN-13:9780521898065 Cambridge University Press, 2009.

\bibitem{Gennes:76}
De Gennes in C.~Domb and M.S.~Green eds., Phase Transition and
Critical Phenomena, {\bf 3}, Academic Press, London, 1976.

\bibitem{Go:67}
M.~G\^{o}, Statistical Mechanics of Biopolymers and its application
to the melting transition of polynucleotides, \textit{J. Phys. Soc.
Jpn}, {\bf 23} (1967), 597-608.

\bibitem{Lesk:74}
A.M. Lesk, A combinatorial study of the effects of admitting non-Watson-Crick
base pairing and of base composition on the helix-forming potential of
polynucleotides of random sequences, \textit{J. Theor. Biol.} {\bf 44} (1974), 7-17.

\bibitem{Lyngso}
R.B. Lyngs{\o} and  C.N.S. Pedersen, RNA Pseudoknot Prediction in Energy Based Models,
\textit{Journal of Computational Biology} {\bf 7} (2000), 409-427.

\bibitem{Matt:99}
D.H. Mathews, J. Sabina, M. Zucker and D.H. Turner, Expanded
sequence dependence of thermodynamic parameters improves prediction
of RNA secondary structure, \textit{J. Mol. Biol.} {\bf 288}(5)
(1999), 911-940.

\bibitem{MN:01}
Markus E.~Nebel, Combinatorial Properties of RNA secondary structures,
\textit{J. Comp. Biol.}, {\bf 9} (2001), 541-573.

\bibitem{Pipas:75}
J.M. Pipas and J.E. McMahon, Method for predicting RNA secondary structure,
\textit{Proc. Nat. Acad. Sci. U.S.A.} {\bf 72} (1975), 2017-2021.

\bibitem{Stanley:80}
R.P. Stanley, Differentiably finite power series, \textit{Eur. J.
Combinator.} {\bf 1} (1980), 175-188.

\bibitem{Viennot:02}
X.G. Viennot, A Strahler bijection between Dyck paths and planar
trees, \textit{Discr. Math} {\bf 246} (2002), 317-329.


\bibitem{Waterman:78}
M.~S.~Waterman, Secondary Structure of Single-Stranded Nucleic
Acids, \textit{Adv. in Math. Suppl. Stud.} {\bf 1} (1978), 167-212.

\bibitem{Zucker:86}
M. Zucker, RNA folding prediction: The continued need for
interaction between biologists and mathematicians, \textit{Lectures
on Mathematics in the Life Sciences} \textbf{17}, 87-124.
Springer-Verlag, 1986.

\end{thebibliography}
\end{document}